\documentclass[11pt]{amsart}
\usepackage{amsmath, amsthm}
\usepackage{amssymb}
\usepackage{amscd}
\usepackage{url}
\usepackage{graphicx}
\usepackage{color}

\numberwithin{equation}{section}

\theoremstyle{plain}

\newtheorem{theorem}[subsection]{Theorem}
\newtheorem{proposition}[subsection]{Proposition}
\newtheorem{lemma}[subsection]{Lemma}

\newcommand{\Q}{\mathbb{Q}}
\newcommand{\Z}{\mathbb Z}

\newcommand{\R}{\mathbb{R}}

\title{Lattice points in the 3-dimensional torus}
\author{Fernando Chamizo and Dulcinea Raboso}
\address{Departamento de Matem\'{a}ticas and ICMAT\\
Facultad de Ciencias
\\
Universidad Aut\'{o}no\-ma de Madrid
\\
28049 Madrid. Spain}
\email{fernando.chamizo@uam.es}
\email{dulcinea.raboso@uam.es}
\thanks{The first author is partially supported by the grant
MTM2011-22851 from the Ministerio de Ciencia e Innovaci\'{o}n (Spain).}

\begin{document}
\begin{abstract}
We prove the exponent $4/3$ for the lattice point discrepancy of a torus in $\R^3$ (generated by the rotation of a circle around the $z$ axis). The exponent comes from a diagonal term and it seems a natural limit for any approach based solely on classical methods of exponential sums.
The result extends to other solids in $\R^3$ related to the torus.
\end{abstract}
\maketitle

\pagestyle{myheadings} 
\renewcommand{\proofname}{Proof}
\renewcommand{\refname}{References}

\section{Introduction}
Consider a fixed compact solid body $\mathbb{B}\subset\R^3$. The study of
\[
\#\big\{\vec{n}\in\Z^3\;:\; R^{-1}\vec{n}\in\mathbb{B}\big\},
\qquad
R>1,
\]
is a basic problem in lattice point theory. Under very general regularity conditions, this quantity is well approximated by $|\mathbb{B}|R^3$ where $|\mathbb{B}|$ denotes the volume of $\mathbb{B}$. Assuming regularity and convexity, meaning positive Gaussian curvature, the error in this approximation, the so-called lattice point discrepancy, is asymptotically smaller than the scaled area of the boundary, namely it is~$O\big(R^{\gamma}\big)$ for some $\gamma<2$. Many authors have devoted their efforts to give a partial answer for the natural question of determining the infimum of the valid values of $\gamma$. Commonly these results depend on subtle exponential sum techniques. The question remains open even for simple bodies like the sphere (see the survey \cite{IvKrKuNo} and the new results in \cite{guo}).

The non-convex case has also attracted the interest of researchers especially in the last decade (see for instance \cite{peter} \cite{kratzel}  \cite{kratzel2} \cite{nowak}  \cite{nowakz} \cite{guoz}).
A notable difference is that sometimes the main term has to be complemented with a secondary main term coming, in some way, from the points of vanishing curvature. Note for instance that if $\mathbb{B}$ is the cube $[-1,1]^3$ with smoothed corners and edges (preserving the flatness of the faces) then $\#\big\{\vec{n}\in\Z^3\;:\; R^{-1}\vec{n}\in\mathbb{B}\big\}$ counts more than $R^2$ on the boundary when $R\in\Z^+$.

A particularly symmetric example, considered in \cite{nowak}, is the solid torus
\begin{equation}\label{torus}
\mathbb{T}=
\Big\{(x,y,z)\in\R^3\;:\;
\big(\rho'-\sqrt{x^2+y^2}\big)^2+z^2\le\rho^2\Big\}
\end{equation}
where $0<\rho<\rho'$ are fixed constants.
Its volume is $2\pi^2\rho^2\rho'$ but the natural main term to approximate the number of lattice points in the $R$-scaled torus is given by
\begin{equation}\label{main}
\mathcal{M}(R)=
2\pi^2\rho^2\rho'R^3
+4\pi\rho\rho'R^{2}
\sum_{n=1}^\infty
\frac{J_1(2\pi \rho Rn)}{n}
\end{equation}
where $J_1$ is the Bessel function.
The second main term has an oscillatory character and it is well approximated by $R^{3/2}$ times a $\rho^{-1}$-periodic function. In particular, when $R\in\rho^{-1}\Z^+$, up to a negligible error term, \eqref{main} can be substituted by $2\pi^2\rho^2\rho'R^3+C\rho^{1/2}\rho'R^{3/2}$ with $C$ a certain constant.

\

Our main result bounds the lattice point discrepancy when $\mathcal{M}(R)$ is taken as the main term.

\begin{theorem}\label{mainth}
With the notation of \eqref{torus} and \eqref{main}, consider
\[
\mathcal{N}(R)=
\#\big\{\vec{n}\in\Z^3\;:\; R^{-1}\vec{n}\in\mathbb{T}\big\},
\qquad\text{and}\qquad
\mathcal{E}(R)=
\mathcal{N}(R)-\mathcal{M}(R).
\]
Then we have
\[
\mathcal{E}(R)=
O\big(R^{4/3+\epsilon}\big)
\qquad\text{for every}\quad\epsilon>0.
\]
\end{theorem}

The previously best known result is due to W.G. Nowak \cite{nowak} who obtained $\mathcal{E}(R)=
O\big(R^{11/8+\epsilon}\big)$. He wrote $\mathcal{M}(R)$ as a series depending on elementary functions but it is equivalent to our statement after substituting the asymptotic formula for $J_1$.

The exponent $4/3$ is better than the best known result for general convex bodies of rotation. This should be stressed  because in principle the non-convex case is more difficult from the analytic point of view. The exponent comes from a diagonal term and then it seems unlikely to be improved in the context of the classical approaches based on exponential sums.
If one could count with precision (beyond the limits of the harmonic analysis) points close to the boundary of the scaled theory then one could parallel the arguments of \cite{ChIw} (improved in \cite{heath}) or \cite{ChCrUb} to go beyond $4/3$. The multiplicative harmonics (Dirichlet characters) and the automorphic harmonics employed in these papers apparently cannot be adapted to the torus.

\

The structure of the paper is as follows: In \S2 we apply the Poisson summation formula to obtain formulas for $\mathcal{N}(R)$ and $\mathcal{E}(R)$. In \S3 an application of the stationary phase principle allows to interpret the formula for $\mathcal{E}(R)$ as an exponential sum.
In \cite{nowak} this aim is reached indirectly thanks to the truncated Hardy-Vorono\"{\i} formula for the circle problem \cite{ivicb} (the circles appear because each horizontal section of $\mathbb{T}$ is a corona). The advantages of our approach are that it is applicable to other problems, it is self-contained and  it reveals the main term in a more transparent way.
We devote \S4 to estimate some exponential sums. Part of the work is done in \cite{ChCr} and \cite{ChIw}. In \S5 we combine the results of the previous sections to get Theorem~\ref{mainth}.
Finally, in \S6 we consider some extensions of our method.

\

In the following sections we assume $\rho'=1$. It can be done without loss of generality because  $\mathcal{N}(R)$ is clearly invariant by $(R,\rho,\rho')\mapsto (\lambda^{-1}R,\lambda\rho,\lambda\rho')$.
We use the standard notation $e(x)=e^{2\pi i x}$ and $\epsilon$ denotes in this paper an arbitrarily small quantity, not necessarily the same each time. The constants involved in~$O$ and~$\ll$ (notations that we consider equivalent) may depend on $\epsilon$ and on $\rho$.

\

\section{The application of the Poisson summation formula}

Let $\chi$ be the characteristic function of $\mathbb{T}$
\[
\chi(\vec{x})=1\quad\text{ if }\vec{x}\in\mathbb{T}
\qquad\text{and}\qquad
\chi(\vec{x})=0\quad\text{ if }\vec{x}\in\R^3-\mathbb{T}.
\]
A formal application of the Poisson summation formula would give
\[
\mathcal{N}(R)
=
\sum_{\vec{n}\in\Z^3}\chi\big(R^{-1}\vec{n}\big)
=
R^3\sum_{\vec{n}\in\Z^3}\widehat{\chi}(R\vec{n})
\]
and it can be checked that $\mathcal{M}(R)$ comes from the terms with $n_1=n_2= 0$ (note that
$\widehat{\chi}(\vec{0})=|\mathbb{T}|=2\pi^2\rho^2$ and see Lemma~\ref{fouriert} below).
In principle this would allow to express directly $\mathcal{E}(R)$ as an oscillatory series. Unfortunately, the hypotheses to apply the Poisson summation formula are not fulfilled and the convergence of the last series  is not assured.
A standard analytic tool to avoid this problem is to introduce some smoothing in $\chi$. We proceed as in Proposition~2.1 of \cite{revo}. For the sake of completeness we include here a proof.
We do not use any special property of $\mathbb{T}$ and in fact the result holds for any smooth compact solid body.

\begin{proposition}\label{poisson}
Given $R>2$ and $\delta= R^{-c}$ with $0<c<1$ a fixed constant, there exists $R-2\delta^{1-\epsilon}<R'<R+2\delta^{1-\epsilon}$ such that
\[
\mathcal N(R)= 2\pi^2 \rho^2R^3+R'^3\sum\eta\big(\delta\|\vec{n}\|\big)\widehat{\chi}(R'\vec{n})+O\big(R^{2+\epsilon}\delta\big)
\]
where the summation is over ${\vec{n}}\in \Z^3-\{\vec{0}\}$.
\end{proposition}

\begin{proof}
Let $\eta\in C_0^\infty(-1,1)$ be a function with $\eta(0)=1$ and such that the Fourier transform of $\eta(\|\vec{x}\|)$ is positive and let $\widehat{\eta}_{\delta}$ be the Fourier transform of $\eta(\delta\|\cdot\|)$. Then, for $k\geq 1$ we have
\[
\int_{\|\vec{t}\|\leq \delta^{1-\epsilon}} \widehat{\eta}_{\delta}(\vec{t})\;d\vec{t}=1+O(\delta^k)
\quad \text{and}\quad
\int_{\|\vec{t}\|\geq \delta^{1-\epsilon}} \widehat{\eta}_{\delta}(\vec{t})\;d\vec{t}=O(\delta^k).
\]
Besides that, we consider the convolution
\[
(\widehat{\eta}_{\delta}*\chi_{R})(\vec{x})
=
\int_{\R^3}\widehat{\eta}_{\delta}(\vec{t})\chi_{R}(\vec{x}-\vec{t})\; d\vec{t},
\]
where $\chi_R(\vec{x})=\chi(R^{-1}\vec{x})$.
If $R_1=R-2\delta^{1-\epsilon}$ and $R_2=R+2\delta^{1-\epsilon}$ we have
\[
(\widehat{\eta}_{\delta}*\chi_{R_1})(\vec{x})\leq \chi_R(\vec{x})+O(\delta^k)
\qquad \text{and} \qquad
(\widehat{\eta}_{\delta}*\chi_{R_2})(\vec{x})\geq \chi_R(\vec{x})+O(\delta^k).
\]
Therefore, for some $R'$ such that $|R'-R|<2\delta^{1-\epsilon}$,
\[
\sum_{\vec{n}\in\Z^3}(\widehat{\eta}_{\delta}*\chi_{R'})(\vec{n}) = \mathcal N(R) + O\big(R^{3}\delta^k\big).
\]
The Poisson summation formula applied to the first term gives
\[
\mathcal N(R)
=
R'^{3}\eta(0)\widehat{\chi}(0) +
R'^{3}\sum_{\vec{0}\neq \vec{n}\in\Z^3}\eta\big(\delta\|\vec{n}\|\big)\widehat{\chi}(R'\vec{n}) +O\big(R^{3}\delta^k\big).
\]
Noting that $R'^3=R^3+O(R^2\delta^{1-\epsilon})$ and taking $k$ large enough we conclude the proof.
\end{proof}

To interpret the sum as an exponential sum it is important to study the oscillation of $\widehat{\chi}$. As a preliminary step, we give a convenient integral formula for $\widehat{\chi}$.

\begin{lemma}\label{fouriert}
The Fourier transform of $\chi$ is a radial function in the first two variables and verifies
$\widehat{\chi}(\xi_1, \xi_2,\xi_3)=
\Psi(\xi_1^2+\xi_2^2,\xi_3)$ for $\xi_1^2+\xi_2^2\neq 0$, where
\[
\Psi(t,u)=
\frac{\rho}{\sqrt{t}}
\int_0^{2\pi}
(1+\rho\sin\theta)\sin\theta
J_1\big(2\pi (1+\rho\sin\theta)\sqrt{t}\big)
e(\rho u\cos\theta)
\; d\theta.
\]
\end{lemma}

\begin{proof}
It is well known that the Fourier transform of the radial function is radial, separating the first two variables, $\widehat{\chi}(\xi_1, \xi_2,\xi_3)=
\Psi(\xi_1^2+\xi_2^2,\xi_3)$ where $\Psi(t,u)=\widehat{\chi}(\sqrt{t},0,u)$. By changing to cylindrical coordinates and the integral representation
\[
J_0(x)
=
\frac{1}{2\pi}\int_0^{2\pi}
\cos(x\cos\theta)\; d\theta,
\]
we have
\[
 \Psi(t,u)=
 \iiint_{\mathbb{T}}e(-x\sqrt{t}-zu)\; dxdydz
 =\iint_D
 2\pi rJ_0(2\pi r\sqrt{t})e(-zu)\; drdz,
\]
where $D$ is the disc $D=\big\{(r,z)\in\R^2\;:\; (1-r)^2+z^2\le \rho^2\big\}$.

The function in the integral is the derivative with respect to the variable $r$ of
$f(r,z)=rt^{-1/2}J_1(2\pi r\sqrt{t})e(-zu)$, because $\big(zJ_1(z)\big)'=zJ_0(z)$. Hence, Green's theorem applies to the field $\vec{F}=\big(0,f\big)$ giving the expected result with the parametrization of the boundary $r=1+\rho\sin\theta$, $z=-\rho\cos\theta$
\end{proof}

After the application of the Poisson summation formula, the volume term $2\pi \rho^2R^3$ becomes apparent. It turns out that the secondary main term comes from the part of the sum with $\vec{n}$ in the $z$-axis.

\begin{proposition}\label{error_s}
With the notation of Proposition~\ref{poisson}
\[
\mathcal{E}(R)= R'^3\sideset{}{'}\sum\eta\big(\delta\|\vec{n}\|\big)\widehat{\chi}(R'\vec{n})+O\big(R^{2+\epsilon}\delta\big)
\]
where the summation is over $\vec{n}\in \Z^3$ such that $n_1^2+n_2^2\ne 0$.
\end{proposition}

\begin{proof}
Our proof starts with the observation that for small arguments the Bessel function satisfies $J_1(z)\sim z/2$ so Lemma \ref{fouriert} shows by the continuity of $\widehat{\chi}$ that the sum in Proposition~\ref{poisson} restricted to the terms with $n_1=n_2=0$ can be written as
\[
2\pi\rho^2 R'^{3}
\sum_{n=1}^\infty
\eta\big(\delta n\big)
\int_0^{2\pi}
\sin^2\theta
\cos(2\pi \rho R' n\cos\theta)
\; d\theta,
\]
where the integral equals to $(\rho R' n)^{-1} J_1(2\pi \rho R'n)$.
Combining $|R'-R|\ll \delta^{1-\epsilon}$ with the asymptotic formula, now for large values,
\begin{equation}\label{eq:asymJ1}
J_1(2\pi z)=\frac{1}{\pi\sqrt{z}}\cos\big(2\pi z-\frac{3\pi}{4}\big)+O\big(z^{-3/2}\big),
\end{equation}
the error made by taking $R$ instead of $R'$ comes from the sums
\[
R^{3/2}\sum_{n>\delta^{-1+\epsilon}}n^{-3/2}
\qquad \text{and} \qquad
R^{3/2}\sum_{n<\delta^{-1}}\frac{\big|\eta(\delta n)e(R'n)-e(Rn)\big|}{n^{3/2}}.
\]
Using $\eta(x)=1+O(x)$ and $\big|e(R'n)-e(Rn)\big|\ll \delta n$, both sums are $\ll R^{3/2+\epsilon}\delta^{1/2}$, which completes the proof.
\end{proof}

\section{Preparation of the exponential sum}

In this section we want to express the sum appearing in Proposition~\ref{error_s} as an exponential sum. The basic argument is the application of the stationary phase principle to the integral representation of the Fourier transform given in Lemma~\ref{fouriert}.

\begin{proposition}\label{stationary}
For $t\ge 1$ and  $u\ge 1$
\[
\Psi(t,u)=
\frac{C\rho^{1/2}}{t^{1/4}\ell^{2}}
\Big(
T(t,u)
+O\big(t^{-1/2}\ell^{1/2}\big)
\Big)
+O\big(t^{-5/4}\big)
\]
where $C$ is an absolute constant, $\ell = (t+u^2)^{1/2}$ and
\[
T(t,u)=
(\ell+\rho\sqrt{t})^{1/2}
\cos\big(2\pi (\rho\ell+\sqrt{t})\big)
-
(\ell-\rho\sqrt{t})^{1/2}
\sin\big(2\pi (\rho\ell-\sqrt{t})\big).
\]
\end{proposition}

\begin{proof}
Substituting the asymptotic formula \eqref{eq:asymJ1} we can rewrite $\Psi(t,u)$ as
\[
\frac{\rho}{\pi t^{3/4}}
\int_0^{2\pi}
A(\theta)\cos\big(2\pi \sqrt{t}(1+\rho\sin\theta)-\frac{3\pi}{4}\big)
e(\rho u\cos\theta)
\; d\theta + O(t^{-5/4})
\]
where $A(\theta)=(1+\rho\sin{\theta})^{1/2}\sin{\theta}$.
The above integral is the real part of
\[
\frac{e(\sqrt{t}-3/8)}{2}\int_0^{2\pi}
A(\theta)
\Big(
e\big(\rho(\sqrt{t}\sin{\theta}+u\cos{\theta})\big)
+e\big(\rho(\sqrt{t}\sin{\theta}-u\cos{\theta})\big)
\Big)
d\theta.
\]
Write $\ell = (t+u^2)^{1/2}$ and let $0<\theta_{tu}<\pi/2$ such that $\tan \theta_{tu}=\sqrt{t}/u$, so  $\sqrt{t}\sin{\theta}\pm u\cos{\theta}=\pm\ell\cos(\theta\mp \theta_{tu})$. Hence by a linear change of variables and noting that $A(\theta+\pi-\theta_{tu})=A(\theta_{tu}-\theta)$ we have
\[
\Psi(t,u)=
\frac{\rho}{\pi t^{3/4}}\Re\left(e\big(\sqrt{t}-3/8\big)\int_0^\pi A(\theta_{tu}-\theta)e(\rho\ell \cos{\theta})\; d\theta\right)
+ O(t^{-5/4}).
\]
We note that the main contribution to the resulting integral comes from regions of the points $\theta=0$ and $\theta=\pi$ at which the phase $\cos{\theta}$ is stationary. Applying the principle of stationary phase, the integral equals
\[
\rho^{-1/2}\ell^{-1/2} \Big(e(\rho\ell -1/8)A(\theta_{tu})+e(-\rho\ell+1/8)A(\theta_{tu}-\pi)+O(\ell^{-1})\Big).
\]
The proof is completed by noting that
\[
A(\theta_{tu})=(\ell+\rho\sqrt{t})^{1/2}t^{1/2}\ell^{-3/2}\quad \text{and}\quad A(\theta_{tu}-\pi)=-(\ell-\rho\sqrt{t})^{1/2}t^{1/2}\ell^{-3/2}.
\]
\end{proof}

\section{Estimation of the exponential sum}

After partial summation the relevant estimate for the main theorem is embodied in the following result.

\begin{proposition}\label{estimation}
For every   $M\le R^{4/3}$, $N\le R^{2/3}$ and any choice of the sign $\pm$
\[
\sup_{\substack{u<M\\ v<N}}
\sum_{\substack{M\le m<M+u\\ N\le n<N+v}}
r(m)
 e\big(R(\rho\sqrt{m+n^2}\pm \sqrt{m})\big)
=O\big(R^{1/3+\epsilon}M^{1/4}L^{3/4}\big)
\]
where $L=M+N^2$  and $r(m)$ is the number of representations of $m$ as a sum of two squares.
\end{proposition}

The crucial step in the proof of Proposition~\ref{estimation} is a variation of the arguments of \cite{ChCr} that are a generalization of \cite{ChIw} and give a bound for
\begin{equation}\label{def_S}
S=
\sum_{M\le m<2M}
\Big|
\sum_{N\le n<2N}
e(\theta n)
 e\big(R\rho\sqrt{m+n^2}\big)
\Big|
\qquad\text{where $\theta\in\R$}.
\end{equation}

\begin{lemma}\label{fromChCr}
With notation as above, there exists some $D\le N^{2-\epsilon}$ such that for $RD\ll L^{3/2}$ we have
\[
S\ll MN^{1/2}+R^{-1+\epsilon}L^{3/2}M,
\]
and for $RD\gg L^{3/2}$,
\[
S\ll
MN^{1/2}
+R^{1/4+\epsilon}N^{7/6}L^{-1/24}M^{1/2}
+R^{\epsilon}NL^{1/4}M^{1/2}.
\]
\end{lemma}

\begin{proof}
Following the first steps in Lemma~3.1 of \cite{ChIw}, by Cauchy's inequality
\begin{equation}\label{cau}
S^2
\ll
MN^\epsilon
\sum_{n_1,n_2}
\Big|
\sum_{m\asymp M}
 e\Big(R\rho\big(\sqrt{m+n_1^2}-\sqrt{m+n_2^2}\big)\Big)
\Big|,
\end{equation}
where $|n_1-n_2|<N^{1-\epsilon}$. Thus, for a suitable $D\le N^{2-\epsilon}$ and $\omega\in\R$, we get by Fourier inversion and the mean value theorem \cite[\S12.2]{IwKo} \cite[Lemma~7.3]{GrKo}
\[
S^2
\ll
M^2N+
MR^\epsilon
\sum_{y\asymp D}
\Big|
\sum_{x\asymp L}
 e(\omega x)e\big(R\rho(\sqrt{x}-\sqrt{x+y})\big)
\Big|.
\]
This $\omega$ is introduced to complete the summation and, in this way, the range of $x$ does not depend on $y$.

On the one hand, if $RD\ll L^{3/2}$ then $f(x,y)=R\rho(\sqrt{x}-\sqrt{x+y})$ is monotonic in $x$ and its partial derivative satisfies $\partial_1f\asymp RDL^{-3/2}$. Therefore by the first derivative test (see Theorem~2.1 of \cite{GrKo}) the inner sum is $\ll R^{-1}D^{-1}L^{3/2}$.

If $RD\gg L^{3/2}$ we apply the $B$-process of the van der Corput method (Lemma~3.6 of \cite{GrKo}) to transform the sum. This is done in Lemma~3.1 of \cite{ChIw}. We deduce in this case (cf. Lemma~3.3 of \cite{ChCr})
\[
S^2
\ll
M^2N+
R^{-1/2+\epsilon}D^{-1/2}L^{5/4}M\big|T_{DL}\big|,
\]
where $T_{DL}$ is the exponential sum
\[
T_{DL}=\sum_{y\asymp D}
\Big|
\sum_{x\asymp RDL^{-3/2}}
e\big(g(x,y)\big)
\Big|,
\]
with $g(x,y)=f(\alpha(x,y),y)-x\alpha(x,y)$ and $\alpha=\alpha(x,y)$ implicitly defined as $\partial_1 f(\alpha(x,y),y)=x$.

Finally, Proposition~3.6 of \cite{ChCr} with $p=q=1/2$ gives
\[
T_{DL}
\ll RD^{5/3}L^{-4/3}+R^{1/2+\epsilon}D^{3/2}L^{-3/4},
\]
and the proof is complete.
\end{proof}

\begin{proof}[Proposition \ref{estimation}]
By the completing sum technique,
\begin{equation}\label{compl}
\sup_{\substack{u<M\\ v<N}}
\sum_{\substack{M\le m<M+u\\ N\le n<N+v}}
r(m)
 e\big(R(\rho\sqrt{m+n^2}\pm \sqrt{m})\big)
\ll R^{\epsilon}S
\end{equation}
The result follows from Lemma \ref{fromChCr} by dividing into two cases, $N^2<M$ and $N^2\geq M$, which give rise to $L\asymp M$ and $L\asymp N^2$, respectively.
\end{proof}

\section{End of the proof}

To prove the main theorem we have just to combine the results of the previous sections.

\begin{proof}[Proof of Theorem \ref{mainth}]
By Proposition~\ref{error_s} and Lemma~\ref{fouriert}, taking $\delta=R^{-2/3}$ if we prove that
\[
R^2\sum_{n=-\infty}^\infty\sum_{m=1}^\infty r(m)\eta\big(R^{-2/3}\sqrt{m+n^2}\big)\Psi(R^2m,Rn)
=
O\big(R^{1/3+\epsilon}\big)
\]
the assertion follows.
The contribution of $n=0$ is absorbed by the error term, and the symmetry in $n$ allows to consider the first summation restricted to $n\geq 1$.
Now Proposition~\ref{stationary} shows that the above is equivalent to
\begin{equation}\label{eq:theo}
\sum_{\substack{m<R^{4/3}\\ n<R^{2/3}}} H(m,n) r(m)e\big(R(\rho\sqrt{m+n^2}\pm \sqrt{m})\big)
=
O\big(R^{1/3+\epsilon}\big)
\end{equation}
where $H(m,n)= h(m,n)\eta\big(R^{-2/3}\sqrt{m+n^2}\big)$ with
\[
h(m,n)=\frac{\big(\sqrt{m+n^2}\pm\rho\sqrt{m}\big)^{1/2}}{m^{1/4}(m+n^2)}
\]
because the error terms in this proposition are absorbed again by the error term.

We can write
\begin{equation}\label{eq:hF}
h(m,n)=\frac{F(n^2/m)}{m^{1/4}(m+n^2)^{3/4}}
\qquad
\text{with }\quad
F(x)=
\Big(1\pm \frac{\rho}{\sqrt{x+1}}\Big)^{1/2}.
\end{equation}
A calculation proves
\[
F'=\mp\frac{\rho}{4}
\big((x+1)^{3/2}F\big)^{-1}
\quad\text{and}\quad
F''=
\pm\frac{3\rho}{8}
\big((x+1)^{5/2}F\big)^{-1}
-\frac{\rho^2}{16}
\big((x+1)F\big)^{-3}.
\]
Hence for $x\in\R^+$ and any choice of the $\pm$ sign
\[
F'(x)
\ll
(x+1)^{-3/2}
\qquad\text{and}\qquad
F''(x)
\ll
(x+1)^{-5/2}
\]

Dividing the summation in \eqref{eq:theo} into dyadic intervals, we can restrict ourselves to $M\leq m<2M$, $N\leq n<2N$ with $M,N^2\ll R^{4/3}$.
Using the previous bounds, in these ranges we have
\[
F
\big(
\frac{n^2}{m}
\big)
\ll 1,
\quad
F'
\big(
\frac{n^2}{m}
\big)
\ll M^{3/2}L^{-3/2}
\quad\text{and}\quad
F''
\big(
\frac{n^2}{m}
\big)
\ll M^{5/4}L^{-5/4}
\]
and $h$ satisfies
\begin{equation*}
\begin{aligned}
h(m,n)\ll M^{-1/4}L^{-3/4},&\qquad
\partial_1h(m,n)\ll M^{-5/4}L^{-3/4}\\
\partial_2h(m,n)\ll M^{-1/4}N^{-1}L^{-3/4},&\qquad
\partial_1\partial_2h(m,n)\ll M^{-5/4}N^{-1}L^{-3/4}\\
\end{aligned}
\end{equation*}
with $L=M+N^2$. The same bounds hold for $H$ because the factor  involving $\eta$ can be seen as $\phi(m/R^{4/3},n/R^{2/3})$ with $\phi$ a smooth function.

Abel's summation formula in both variables reads (see Lemma $\alpha$ of \cite{titchmarsh} and Lemma~4 of \cite{KrNo} for other forms of partial summation)
\begin{equation*}
\begin{aligned}
\sum_{a\le m\le b}
\sum_{c\le n\le d}
a_{mn} & H(m,n)
=\;
H(b,d)
A(b,d)
-\int_a^b
A(t,d)
\partial_1 H(t,d)\; dt
\\
&
-\int_c^d
A(b,u)
\partial_2 H(b,u)\; du
+\int_a^b\int_c^d
A(t,u)
\partial_{12} H(t,u)\; dudt
\end{aligned}
\end{equation*}
where
\[
 A(t,u)=
\sum_{a\le m\le t}
\sum_{c\le n\le u}
a_{mn}
\qquad\text{where $a<b$, $c<d$, are positive integers.}
\]

In our case, when applied to the left hand side of \eqref{eq:theo}, after the dyadic subdivision, we get that it is bounded by
\[
R^{\epsilon}M^{-1/4}L^{-3/4}
\sup_{\substack{u<M\\ v<N}}
\sum_{\substack{M\le m<M+u\\ N\le n<N+v}}
r(m)
 e\big(R(\rho\sqrt{m+n^2}\pm \sqrt{m})\big)
\]
and Proposition \ref{estimation} yields the assertion of the theorem.
\end{proof}

\section{Some generalizations}

Given  $A\in\text{GL}_3(\Q)$ we consider the lattice point problem for $A\mathbb{T}$ i.e.,  we are interested in approximating
\begin{equation}\label{N_A}
\mathcal{N}_A(R)=
\#\big\{\vec{n}\in\Z^3\;:\; R^{-1}\vec{n}\in A\mathbb{T}\big\}.
\end{equation}
The main term is the following variation of \eqref{main}
\begin{equation}\label{M_A}
\mathcal{M}_A(R)=
2\pi^2\rho^2\rho'|\det(A)|R^3
+4\pi\rho\rho'r_A^{-1}|\det(A)| R^{2}
\sum_{n=1}^\infty
\frac{J_1(2\pi Rr_A\rho n)}{n}
\end{equation}
where $r_A=\min\big\{r>0\;:\;(r,0,0)\in A^t\Z^3\big\}$.

With this notation, Theorem~\ref{mainth} generalizes to

\begin{theorem}\label{mainth2}
With the definitions \eqref{N_A} and \eqref{M_A}, for a fixed $A\in\text{\rm GL}_3(\Q)$ the lattice point discrepancy
$\mathcal{E}_A(R)=
\mathcal{N}_A(R)-\mathcal{M}_A(R)$ satisfies
\[
\mathcal{E}_A(R)=
O\big(R^{4/3+\epsilon}\big)
\qquad\text{for every}\quad\epsilon>0.
\]
\end{theorem}

\begin{proof}
As before, the invariance
$(R,\rho,\rho')\mapsto (\lambda^{-1}R,\lambda\rho,\lambda\rho')$
allows us to assume $\rho'=1$ choosing $\lambda=1/\rho'$. On the other hand, by the invariance
$(R,A)\mapsto (\mu^{-1}R,\mu A)$, taking $\mu$ as the common denominator of the entries of $A$ we can assume that $A$ is an integral matrix. The reflection through the $xy$-plane $S:(x,y,z)\mapsto (x,y,-z)$ preserves $\mathbb{T}$, in particular $\mathcal{N}_{AS}(R)=\mathcal{N}_{A}(R)$ and we can also assume that $\det(A)>0$, changing $A$ by $AS$ if necessary.

Note that $\chi_A(\vec{x})=\chi(A^{-1}\vec{x})$ is the characteristic function of $A\mathbb{T}$. Applying Proposition~\ref{poisson} with $\chi_A$ instead of $\chi$ and using in the proof $\eta\big(\|A^t\vec{x}\|\big)$ instead of $\eta\big(\|\vec{x}\|\big)$, we have
\begin{equation}\label{poisson_A}
\mathcal N_A(R)= 2\pi^2 \rho^2|A|R^3+
|A|R'^3
\sum_{\vec{n}\in\Z^3-\{\vec{0}\}}\eta\big(\delta\|A^t\vec{n}\|\big)\widehat{\chi}(R'A^t\vec{n})+O\big(R^{2+\epsilon}\delta\big)
\end{equation}
for $|R'-R|\ll \delta^{1-\epsilon}$,
where we write $|A|=\det(A)$ and we have employed $\widehat{\chi}_A(\vec{\xi})=|A|\widehat{\chi}(A^t\vec{\xi})$ by the properties of the Fourier transform.

Taking the limit $\xi_1^2+\xi_2^2\to 0$ in Lemma~\ref{fouriert}, using $J_1(z)\sim z/2$ as in Proposition~\ref{error_s}, we have
\[
 \widehat{\chi}(0,0,u)=
 2\pi\rho^2
 \int_0^{2\pi}\sin^2\theta\cos(2\pi\rho u\cos\theta)\; d\theta
 =2\pi\rho u^{-1}J_1(2\pi \rho u).
\]
As $\big\{A^t\vec{n}\;:\; \vec{n}\in\Z^3\big\}$ is a lattice, the elements of this set lying on the $z$-axis are  $\big\{(0,0,r_An)\big\}_{n\in\Z}$. Its contribution to the
sum in \eqref{poisson_A} is
\[
\sum_{n\in\Z-\{0\}}
\eta\big(\delta r_A|n|\big)\widehat{\chi}(0,0,R'r_An)
=
4\pi\rho
\sum_{n=1}^\infty
\eta\big(\delta r_A|n|\big)
\frac{J_1(2\pi R'r_A\rho n)}{R'r_An}.
\]
As shown in Proposition~\ref{error_s}, $R'$ can be substituted by $R$ and $\eta$ by $1$ with a negligible error term, and once we multiply by the
$|A|R'^3$ factor in \eqref{poisson_A} we get the secondary main term in \eqref{M_A} with $\rho'=1$. Then the analog of Proposition~\ref{error_s} becomes
\begin{equation}\label{E_A}
\mathcal{E}_A(R)= |A|R'^3
\sideset{}{'}\sum
\eta\big(\delta\|A^t\vec{n}\|\big)
\widehat{\chi}(R'A^t\vec{n})+O\big(R^{2+\epsilon}\delta\big)
\end{equation}
where the summation is restricted to $\vec{n}\in\Z^3$ such that $A^t\vec{n}$ does not lie on the $z$-axis.

Let $q\in\Z^+$ and $B$ an integral matrix such that $qI_3=BA^t$. As $A$ is itself integral one can take $q=|A|$ and $B$ the cofactor matrix of $A$. With this definition
\[
 \big\{A^t\vec{n}\;:\;\vec{n}\in\Z^3\big\}
=
 \big\{\vec{m}\in\Z^3\;:\;\vec{m}\in A^t\Z^3\big\}
=
 \big\{\vec{m}\in\Z^3\;:\;B\vec{m}\in q\Z^3\big\}.
\]
Then in \eqref{E_A} we can replace $A^t\vec{n}$ by $\vec{m}\in\Z^3$ not lying on the $z$-axis and restricted to some congruence classes modulo~$q$. Let $\vec{c}$ be a representative of the class giving the biggest contribution in \eqref{E_A}. Then, renaming $R'$ as $R$, we have to prove, for some choice of $\delta$,
\[
 R^3
 \Big|
 \sum_{\substack{\vec{m}\equiv\vec{c}\ (q) \\ m_1^2+m_2^2\ne 0}}
\eta\big(\delta\|\vec{m}\|\big)
\widehat{\chi}(R\vec{m})
 \Big|
+R^{2+\epsilon}\delta
=
O\big(R^{4/3+\epsilon}\big).
\]
The steps of the proof of Theorem~\ref{mainth} given in the previous section can be completed to obtain this with $\delta=R^{-2/3}$ if we assume Proposition~\ref{estimation} with $n$ restricted to $n\equiv c_3\ (q)$ and $r(m)$ replaced by
\[
 r^*(m)=
 \#\big\{
 (m_1,m_2)\in\Z^2\;:\; m=m_1^2+m_2^2,\ m_1\equiv c_1\ (q),\ m_2\equiv c_2\ (q)
 \big\}.
\]
To justify this assumption, it is enough to note that  the bound $r_*(m)\le r(m)\ll m^\epsilon$ is enough to get the analog of \eqref{compl} but now with $S$ as in \eqref{def_S} but imposing $n\equiv c_3\ (q)$. On the other hand, in \eqref{cau} the corresponding conditions $n_1\equiv n_2\equiv c_3\ (q)$ can be dropped trivially by positivity. Once the congruence conditions are ruled out, the rest of proof applies.
\end{proof}

If $A\in\text{GL}_3(\R)$ and $A$ has not rational entries, the method does not work in general because we cannot mix variables in an arithmetic way and this is a key step in the estimation of the exponential sums appearing in the problem. For instance, note that we have used smoothed forms of the formal identity
\[
 \sum_{n_1}
 \sum_{n_2}
 \sum_{n_3}
 \Psi(n_1^2+n_2^2,n_3)
 =
 \sum_m\sum_n
 r(m)\Psi(m,n)
\]
that does not have an analog if $n_1\mapsto \alpha^{1/2}n_1$ with $\alpha\not\in\Q$ because the fractional part of $\alpha n_1^2+n_2^2$ is dense in $[0,1]$.

An exception is the case of the lattice point problem associated to
\[
 \mathbb{T}_Q=
\Big\{(x,y,z)\in\R^3\;:\;
\big(\rho'-\sqrt{Q(x,y)}\big)^2+z^2\le\rho^2\Big\}
\]
where $Q$ is a positive definite quadratic form with rational coefficients. Geometrically it is a torus with horizontal sections given by elliptic annuli.
If $\chi_Q$ is the characteristic function of $\mathbb{T}_Q$ then using the properties of the Fourier transform one deduces, with the notation of Lemma~\ref{fouriert},
\begin{equation}\label{q_chi}
\widehat{\chi}_Q(\vec{\xi})=
d^{-1/2}\Psi(Q^*(\xi;\xi_2),\xi_3)
\end{equation}
where $d$ is the determinant of $Q$ and $Q^*$ is the quadratic form having as matrix the inverse of the matrix of $Q$. Using a rational homothecy affecting to the two first variables (this transformation is covered by Theorem~\ref{mainth2}) we can assume that $Q^*$ has integral coefficients and then we have
\[
 \sum_{n_1}
 \sum_{n_2}
 \sum_{n_3}
 \Psi\big(Q^*(n_1,n_2),n_3\big)
 =
 \sum_m\sum_n
 r_{Q^*}(m)\Psi(m,n)
\]
where $r_{Q^*}(m)$ is the number of representations of $m$ by the binary quadratic form $Q^*$. The method in the proof of Theorem~\ref{mainth} applies just changing $r(m)$ by $r_{Q^*}(m)$ and it does not affect to the estimation of the exponential sum because $r_Q(m)\ll m^\epsilon$ and can be extracted in \eqref{compl}.

On the other hand, by \eqref{q_chi} the main term \eqref{main} has to be multiplied by~$d^{-1/2}$ and we obtain
\[
 \#\big\{\vec{n}\in\Z^3\;:\; R^{-1}\vec{n}\in\mathbb{T}_Q\big\}
 =
 d^{-1/2}\mathcal{M}(R)+ O\big(R^{4/3+\epsilon}\big)
 \qquad\text{for every}\quad\epsilon>0.
\]
This result also holds if $\mathbb{T}_Q$ is replaced by $A\mathbb{T}_Q$ and $\mathcal{M}(R)$ by $\mathcal{M}_A(R)$ when $A\in\text{GL}_3(\Q)$.


\begin{thebibliography}{IKKN06}

\bibitem[CC12]{ChCr}
F.~Chamizo and E.~Crist{\'o}bal.
\newblock The sphere problem and the {$L$}-functions.
\newblock {\em Acta Math. Hungar.}, 135(1-2):97--115, 2012.

\bibitem[CCU09]{ChCrUb}
F.~Chamizo, E.~Crist{\'o}bal, and A.~Ubis.
\newblock Lattice points in rational ellipsoids.
\newblock {\em J. Math. Anal. Appl.}, 350(1):283--289, 2009.

\bibitem[Cha98]{revo}
F.~Chamizo.
\newblock Lattice points in bodies of revolution.
\newblock {\em Acta Arith.}, 85(3):265--277, 1998.

\bibitem[CI95]{ChIw}
F.~Chamizo and H.~Iwaniec.
\newblock On the sphere problem.
\newblock {\em Rev. Mat. Iberoamericana}, 11(2):417--429, 1995.

\bibitem[GK91]{GrKo}
S.~W. Graham and G.~Kolesnik.
\newblock {\em van der {C}orput's method of exponential sums}, volume 126 of
  {\em London Mathematical Society Lecture Note Series}.
\newblock Cambridge University Press, Cambridge, 1991.

\bibitem[Guo12]{guo}
J.~Guo.
\newblock On lattice points in large convex bodies.
\newblock {\em Acta Arith.}, 151(1):83--108, 2012.

\bibitem[Guo13]{guoz}
J.~Guo.
\newblock Lattice points in rotated convex domains.
\newblock arXiv:1303.4137, 2013.

\bibitem[HB99]{heath}
D.~R. Heath-Brown.
\newblock Lattice points in the sphere.
\newblock In {\em Number theory in progress, {V}ol. 2
  ({Z}akopane-{K}o\'scielisko, 1997)}, pages 883--892. de Gruyter, Berlin,
  1999.

\bibitem[IK04]{IwKo}
H.~Iwaniec and E.~Kowalski.
\newblock {\em Analytic number theory}, volume~53 of {\em American Mathematical
  Society Colloquium Publications}.
\newblock American Mathematical Society, Providence, RI, 2004.

\bibitem[IKKN06]{IvKrKuNo}
A.~Ivi{\'c}, E.~Kr{\"a}tzel, M.~K{\"u}hleitner, and W.~G. Nowak.
\newblock Lattice points in large regions and related arithmetic functions:
  recent developments in a very classic topic.
\newblock In {\em Elementare und analytische {Z}ahlentheorie}, Schr. Wiss. Ges.
  Johann Wolfgang Goethe Univ. Frankfurt am Main, 20, pages 89--128. Franz
  Steiner Verlag Stuttgart, Stuttgart, 2006.

\bibitem[Ivi03]{ivicb}
A.~Ivi{\'c}.
\newblock {\em The {R}iemann zeta-function}.
\newblock Dover Publications Inc., Mineola, NY, 2003.
\newblock Theory and applications, Reprint of the 1985 original [Wiley, New
  York; MR0792089 (87d:11062)].

\bibitem[KN92]{KrNo}
E.~Kr{\"a}tzel and W.~G. Nowak.
\newblock Lattice points in large convex bodies. {II}.
\newblock {\em Acta Arith.}, 62(3):285--295, 1992.

\bibitem[Kr{\"a}02a]{kratzel2}
E.~Kr{\"a}tzel.
\newblock Lattice points in some special three-dimensional convex bodies with
  points of {G}aussian curvature zero at the boundary.
\newblock {\em Comment. Math. Univ. Carolin.}, 43(4):755--771, 2002.

\bibitem[Kr{\"a}02b]{kratzel}
E.~Kr{\"a}tzel.
\newblock Lattice points in three-dimensional convex bodies with points of
  {G}aussian curvature zero at the boundary.
\newblock {\em Monatsh. Math.}, 137(3):197--211, 2002.

\bibitem[Now08a]{nowak}
W.~G. Nowak.
\newblock The lattice point discrepancy of a torus in {$\mathbb R^3$}.
\newblock {\em Acta Math. Hungar.}, 120(1-2):179--192, 2008.

\bibitem[Now08b]{nowakz}
W.~G. Nowak.
\newblock On the lattice discrepancy of bodies of rotation with boundary points
  of curvature zero.
\newblock {\em Arch. Math. (Basel)}, 90(2):181--192, 2008.

\bibitem[Pet02]{peter}
M.~Peter.
\newblock Lattice points in convex bodies with planar points on the boundary.
\newblock {\em Monatsh. Math.}, 135(1):37--57, 2002.

\bibitem[Tit]{titchmarsh}
E.~C. Titchmarsh.
\newblock On {E}pstein's {Z}eta-{F}unction.
\newblock {\em Proc. London Math. Soc.}, S2-36(1):485.

\end{thebibliography}

\end{document}